\documentclass{article}
\usepackage[top=50pt,bottom=60pt,left=70pt,right=70pt]{geometry}
\usepackage{amsmath}
\usepackage{amssymb}
\usepackage{amsthm}
\usepackage{hyperref}
\hypersetup{
  colorlinks   = true, 
  urlcolor     = blue, 
  linkcolor    = blue, 
  citecolor   = red 
}
\usepackage{cleveref}
\usepackage[
backend=biber,
style=alphabetic,
sorting=ynt
]{biblatex}
\usepackage{setspace}
 \newcommand\sect[2]{{
  \left.\kern-\nulldelimiterspace 
  #1 
  \littletaller 
  \right|_{#2} 
  }}
 \newcommand{\littletaller}{\mathchoice{\vphantom{\big|}}{}{}{}}
\newcommand{\nilptwo}{\mathcal{N}_2}
 \newcommand{\sbgrp}[1]{\langle #1\rangle}
 \newcommand{\inner}[2]{\langle #1,#2 \rangle}
\newcommand{\sbmon}[1]{{#1}^*}
\newcommand{\heisen}[1]{H_{#1}(\mathbb{Z})}

 \newcommand{\xx}{\boldsymbol{x}}
 \newcommand{\yy}{\boldsymbol{y}}
 \DeclareMathOperator{\uu}{\boldsymbol{u}}
 \DeclareMathOperator{\vv}{\boldsymbol{v}}
 \DeclareMathOperator{\val}{val}
 
 \newcommand{\N}{\mathbb{N}}
\newcommand{\Z}{\mathbb{Z}}

\newtheorem{Theorem}{Theorem}
\newtheorem{Corollary}[Theorem]{Corollary}
\newtheorem{Proposition}[Theorem]{Proposition}

\theoremstyle{remark}

\title{Sections of Submonoids of Nilpotent Groups}
\author{Doron Shafrir\\
\small{doron.abc@gmail.com}}
\date{}

\begin{document}

\maketitle
\begin{abstract}
    We show that every product of f.g.\ submonoids of a group $G$ is a section of a f.g.\ submonoid of $G{\times}\heisen{5}$, where $\heisen{5}$ is a Heisenberg group. This gives us a converse of a reduction of Bodart, and a new simple proof of the existence of a submonoid of a nilpotent group of class 2 with undecidable membership problem.
\end{abstract}
\section{Background and Overview}
Romankov showed in \cite{roman1999occurence} that the rational subset membership problem in undecidable in some group $G\in\nilptwo$, where $\nilptwo$ is the class of nilpotent groups of class $2$.  Later, \cite{lohrey2015unitri,konig2016knapsack} showed undecidability in the same class $\nilptwo$ for more specific types of rational sets, such as products of subgroups. More recently,  Romankov showed the existence of an $\nilptwo$ group with a submonoid with undecidable membership problem \cite{roman2023undecidability}. All of the above results rely on the negative solution to Hilbert's 10th problem. There are also some recent positive decidability results in nilpotent groups \cite{bodart2024membership,shafrir2024boundedgen}.
In this work, we show that if $G$ is any f.g.\ group and $A\subset G$ is a product of f.g.\ submonoids of $G$, then $A$ is a section of some submonoid $M$ of $G{\times}\heisen{5}$, that is, for some $h\in\heisen{5}$ we have $g\in A$ iff $(g,h)\in M$. The proof uses simple geometric constructions in the plane. Besides being interesting in its own right, this allows us to reduce the membership problem in $A$ to submonoid membership in $G{\times}\heisen{5}$. Using the existence of products of submonoids of some $G\in\nilptwo$ with undecidable membership problem  \cite{konig2016knapsack}, we get a new, simple proof of the existence of an undecidable submonoid of a group $G\in\nilptwo$.
The method of sections was also used in \cite{shafrir2024virtuallyab} to show that any rational subset of $G$ is a section of a submonoid of $G{\times}H$ where $H$ is virtually Abelian. This paper deals exclusively with sections where the second group $H$ is in $\nilptwo$. We first show that for $H=\heisen{3}$ we can get products of 2 submonoids as sections, then we show that for groups of type $H=\heisen{3}/\sbgrp{z^e}$ we can get products of conjugate submonoids as sections. For arbitrary products of submonoids we need a larger group $H=\heisen{5}$. 
In \cite{bodart2024membership}, Bodart showed that submonoid membership in a nilpotent group $G$ can be reduced to membership in products of submonoids in subgroups $H<G$ with $h([H,H])<h([G,G])$. Our result gives a tight converse to this reduction: it gives a reduction from membership in products of submonoids of $G$ to submonoid membership in a larger group $K=G{\times}\heisen{5}$ satisfying $h([K,K])=h([G,G])+1$.

\subsection{Notations}
We define $\N_0=\{0,1,2...\},\N_+=\{1,2,...\}$. If $G$ is a group and $S\subset G$, $\sbgrp{S}$ and $\sbmon{S}$ are the subgroup and submonoid generated by $S$ respectively. $\nilptwo$ denotes the class of nilpotent groups of class 2. All of our groups and submonoids are finitely generated. If $A\subset G{\times}H$ and $h\in H$, we define $\sect{A}{h}=\{g\in G\mid(g,h)\in A\}$, and call $\sect{A}{h}$ the $h$-section of $A$. We view $G,H$ as subgroups of $G{\times}H$, and therefore $A\cap G=\sect{A}{1_H}$. For any $d\in\N_+$, the discrete Heisenberg group $\heisen{2d+1}$ is the group generated by $x_1,...,x_d,y_1,...,y_d,z$ with the relations $[x_i,y_i]=z$ for all $i$, and all other generators commute. In $\heisen{3}$, we use $x,y,z$ instead of $x_1,y_1,z$. For notational convenience, for any vector $v\in\Z^d$, we define $\xx^v=x_1^{v_1}\cdots x_d^{v_d}$ and similarly for $\yy^v$. From $[x_i,y_j]=z^{\delta_{ij}}$ and bi-linearity of the commutator in $\nilptwo$, we get $[\xx^u,\yy^v]=z^{\inner{u}{v}}$ for any $u,v\in\Z^d$ where $\inner{u}{v}$ is the standard inner product. This allows us to use plane geometry in our calculations in $\heisen{5}$.

\section{Sections with \texorpdfstring{$\heisen{3}$}{H3}}
\subsection{Product of a pair of submonoids}The following Proposition also appears in \cite[Lemma 3.1]{roman2023undecidability}.
\begin{Proposition}    Let $t=x^{-1}z$. Let $w\in\{t,y,x\}^{<\omega}$ be any word (without inverses) in $t,y,x$. Then, $\val(w)=y$ iff $w=t^nyx^n$ for some $n\in\N_0$.
\end{Proposition}
\begin{proof}
By projecting $\heisen{3}\rightarrow \heisen{3}/\langle x,z\rangle\simeq\Z$ we see that $y$ appears only once in $w$. Using $[t,x]=1$ and $x^{\pm1}y=yx^{\pm1}z^{\pm1}$ we get:
\[\val(w)=(t^ax^b)y(t^cx^d)=yx^{-a}(xz)^b(x^{-1}z)^cx^d=yx^{-a+b-c+d}z^{b+c}\]
Where $a,b,c,d\ge0$. For the value to be $y$, we must have $b=c=0,a=d$ as needed.
\end{proof}
The existence of elements with the above property implies:
\begin{Theorem}
    If $A,B\le G$ are f.g.\ submonoids, there is a f.g.\ submonoid $C\le G{\times}\heisen{3}$ such that $A\cdot B$ is a section of $C$. Therefore, the membership problem in products of 2 submonoids of $G$ can be reduced to submonoid membership in $G{\times}\heisen{3}$.
\end{Theorem}
For a proof see \cite[Proposition 2]{shafrir2024virtuallyab}.

\subsection{Product of conjugate submonoids}
We first need a technical lemma, about the existence of finite sequences in $\heisen{3}/\sbgrp{z^e}$, whose product has a unique value among all words in the sequence elements. We start with a simple claim about vectors with unique sum.
\begin{Proposition}
\label{prop:uniquesum}
    Let $0<b_1<...<b_n$ such that $b_i>ib_{i-1}$ for every $2\le i\le n$, and let  $v_i=(1,b_i)\in\Z^2$. If $\alpha_i\in\N_0$ satisfy $\sum_{i=1}^n \alpha_iv_i=\sum_{i=1}^n v_i$ then $\alpha_1=\alpha_2=\cdots =\alpha_n=1$. 
\end{Proposition}
\begin{proof}
    We use induction on $n$. The first coordinate gives $\sum \alpha_i=n$. If $\alpha_n=0$ then $\sum \alpha_ib_i\le nb_{n-1}<b_n<\sum b_i$ so the second coordinate of $\sum_{i=1}^n \alpha_iv_i$ is smaller than that of $\sum_{i=1}^n v_i$. If $\alpha_n\ge 2$ then $\sum \alpha_ib_i\ge 2b_n>b_n+nb_{n-1}>\sum b_i$ and again the second coordinate doesn't match. Therefore $\alpha_n=1$ and $\sum_{i=1}^{n-1} \alpha_iv_i=\sum_{i=1}^{n-1} v_i$, and we can use induction.
\end{proof}

\begin{Proposition}\label{prop:uniqueprod}
    For any $n$, there is $e\in\N$ and $h_1,\ldots ,h_n\in \heisen{3}/\sbgrp{z^e}$ such that the word $h_1h_2\ldots h_n$ has a unique value among all words in $\{h_1,..,h_n\}$ (without inverses).
\end{Proposition}
\begin{proof}
    Let $h_i=x^{b_i}y\in \heisen{3}$ where $b_i>ib_{i-1}$ as above. Let $v=h_1\cdot h_2\cdots h_n$ and let $w$ be a word in $h_1,...,h_n$ with value $v$. By applying  \Cref{prop:uniquesum} on the projection of $h_i$ to $\heisen{3}/\sbgrp{z}\simeq\Z^2$, where $\alpha_i$ is the number of appearances of $h_i$ in $w$, we get that each $h_i$ appears exactly once in $w$, so $w=h_{\sigma(1)}\ldots h_{\sigma(n)}$ for some permutation $\sigma\in S_n$. For any $i,j$ we have $[h_i,h_j]=[x^{b_i}y,x^{b_j}y]=z^{b_i-b_j}$ and  $h_ih_j=h_jh_iz^{b_i-b_j}$. Assume that $\sigma\ne Id$. We take $h_n$ and swap it with the next element until it's at the end of $w$, adding a correction factor of $z^{b_n-b_i}$ in each swap, a positive power of $z$. Then swap $h_{n-1}$ with the next letter until $w$ ends with $h_{n-1}h_n$, and so on. Eventually we get $\val(w)=vz^f$ where  $0<f\le\sum_{i<j}(b_j-b_i)$. The maximal value of $f$ is attained for $w=h_n...h_2h_1$. It is now clear that the value $v$ is attained only by the word $h_1\ldots h_n$, and that this uniqueness is maintained in the quotient $\heisen{3}/\sbgrp{z^e}$ if we choose $e>\sum_{i<j}(b_j-b_i)$.
\end{proof}
Geometrically, the exponent $f$ in $z^f=(h_1\cdots h_n)^{-1}(h_{\sigma(1)}\cdots h_{\sigma(n)})$ equals the (signed) area between the paths $0\rightarrow v_1\rightarrow v_1+v_2\rightarrow\cdots\sum v_i$ and $0\rightarrow v_{\sigma(1)}\rightarrow v_{\sigma(1)}+v_{\sigma(2)}\rightarrow\cdots\sum v_i$. It is clear from the definition of $v_i$ that the former path lies below the latter for $\sigma\ne Id$ (with possible partial overlap).

\begin{Proposition}\label{prop:MgMgM}
    If $M\le G$ is a f.g.\ submonoid, $g_1,..,g_n\in G$, then $Mg_1Mg_2\cdots Mg_nM$ is a section of $G{\times}H$ where $H=\heisen{3}/\sbgrp{z^e}$ for some $e\in\N$. 
\end{Proposition}
\begin{proof}
    Let $S\subset G$ be a finite set generating $M$. Let $e\in N$, $h_1,...,h_n\in H=\heisen{3}/\sbgrp{z^e}$ and $v=h_1\cdots h_n$ be as in \Cref{prop:uniqueprod}. Define:
    \[T=S{\times}\{1_H\}\cup\{(g_i,h_i)\mid 1\le i\le n\}\]
    If $(g,v)\in\sbmon{T}$, then by \Cref{prop:uniqueprod} applied on last coordinate, the elements $(g_i,h_i)$ are each used once and in order (but can be far apart). In the first coordinate we get a product of the type $\sbmon{S}g_1\sbmon{S}\cdots g_n\sbmon{S}$ as needed, and it is easy to see any such product is attainable.
\end{proof}

We immediately get:
\begin{Theorem}
\label{thm:conjsubmon}
    Let $G$ be a f.g.\ group, $n\in\N$. The membership problem in products of $n$ conjugate f.g.\ submonoids of $G$ can be reduced to submonoid membership in  $G{\times}H$ where $H=\heisen{3}/\sbgrp{z^e}$ for some $e$ depending on $n$. 
\end{Theorem}
\begin{proof}
We have  $g\in M^{g_1}M^{g_2}\cdots M^{g_n}$ iff $g_1gg_n^{-1}\in Mg_1g_2^{-1}Mg_2g_3^{-1}M\cdots g_{n-1}g_n^{-1}M$ iff $(g_1gg_n^{-1},v)\in\sbmon{T}$ for $T,v$ we get by applying \Cref{prop:MgMgM}.
\end{proof}
We give a simple example of an application of \Cref{thm:conjsubmon}:
\begin{Corollary}
    The membership problem in products of conjugate f.g.\ submonoids in nilpotent groups $G$ with $h([G,G])\le 2$ is decidable, where $h(\cdot)$ is the Hirsch length.
\end{Corollary}
\begin{proof}
   We have to decide  $g\in M^{g_0}M^{g_1}\cdots M^{g_n}$, and by \Cref{thm:conjsubmon} we can reduce this problem to submonoid membership in $K=G{\times}(\heisen{3}/\sbgrp{z^e})$. Since $h([K,K])=h([G,G])\le 2$, submonoid membership is decidable in $K$ by \cite{shafrir2024boundedgen}.
\end{proof}
The proof of \cite{shafrir2024boundedgen} relies on a reduction \cite{bodart2024membership} from submonoid membership in a nilpotent f.g.\ group $G$ to membership in products of submonoids of a subgroup $H$ with $h([H,H])<h([G,G])$. It may seem artificial  to pass through the larger group $G{\times}(\heisen{3}/\sbgrp{z^e})$ before going to a subgroup, and indeed, it is straightforward to generalize Bodart's reduction to deal with products of conjugate submonoids directly. The use of sections is more natural for negative results, as we see next.

\section{Arbitrary products of submonoids as sections}
We now prove that every product of f.g.\ submonoids of a group $G$ is a section of a f.g.\ submonoid of $G{\times}H$, where $H\in\nilptwo$. We give 2 proofs. The first proof is with $H=\heisen{5}{\times}(\heisen{3}/\sbgrp{z^e})$. It relies on \Cref{thm:conjsubmon}, and is conceptually simpler. The second proof has $H=\heisen{5}$, so it is stronger, but it is a bit harder. Since $h([H,H])=1$ in both cases, both are enough for the main theorems in the next section.

\subsection{First proof using conjugate submonoids}

\begin{Theorem}
\label{thm:arbitraryprod2conjprod}
    Let $M_1,...,M_n$ be f.g.\ submonoids of $G$. Then, there is a f.g.\ submonoid $M$ of $G{\times}\heisen{5}$ and elements $h_1,...,h_n\in\heisen{5}$ such that $(M^{h_1}M^{h_2}\cdots M^{h_n})\cap G=M_1\cdots M_n$.
\end{Theorem}
\begin{proof}
    We start with a simple geometric construction. Let $u_i\in\Z^2$, $1\le i\le n$, be the vertices of a convex polygon  such that $\sum_iu_i=(0,0)$. Let $v_i\in\Z^2$ be such that the maximum value of $\inner{x}{v_i}$ over the polygon is attained on the single point $x=u_i$. Geometrically, $v_i$ can be chosen perpendicular to any line whose intersection with the polygon is $\{u_i\}$. Define: $\uu_i=\xx^{u_i}z^{\inner{u_i}{v_i}},h_i=(1_G,\yy^{-v_i})$. Observe that 
    \[\uu_i^{\yy^{-v_j}}=\xx^{u_i}z^{\inner{u_i}{v_i}}[\xx^{u_i},\yy^{-v_j}]=\xx^{u_i}z^{\inner{u_i}{v_i}-\inner{u_i}{v_j}}\]
    By choice of $v_i$, we have $\inner{u_i}{v_i}\ge\inner{u_i}{v_j}$ with equality iff $i=j$, and therefore the exponent of $z$ is non-negative, and $0$ iff $i=j$.
 
 Let $S_i$ be a finite generating set for $M_i$, where WLOG $1_G\in S_i$. Define $T=\bigcup_iS_i\times\{\uu_i\}$ and $M=\sbmon{T}$.  We claim that $(M^{h_1}M^{h_2}\cdots M^{h_n})\cap G=M_1\cdots M_n$. For one side, let $g_i\in M_i$, we need to show $(g_1\cdots g_n,1)\in M$. We can express $g_i=s_{i,1}s_{i,2}\cdots s_{i,l_i}$ where $l_i\in\N,s_{i,j}\in S_i$. Set $l=\max_il_i$. Since $1_G\in S_i$, we may assume $l=l_1\cdots =l_n$ by padding the shorter sequences with $1_G$. We have $(s_{i,j},\uu_i)^{h_i}=(s_{i,j},\xx^{u_i})\in T^{h_i}$, and by taking product over $j$ we get $(g_i,\xx^{lu_i})\in M^{h_i}$.  By taking product now over $i$ we get
 \[(g_1\cdots g_n,\xx^{lu_1}\cdots\xx^{lu_n})\in M^{h_1}\cdots M^{h_n}\]
but the second coordinate is $\xx^{lu_1}\cdots\xx^{lu_n}=\xx^{l\sum u_i}=1$ since $\sum_iu_i=(0,0)$.
Conversely, assume $(g,1)\in M^{h_1}M^{h_2}\cdots M^{h_n}$. We note that in the second coordinate we have values of type $\uu_i^{\yy^{-v_j}}=\xx^{u_i}z^e$ where $e\ge0$ always holds, and $e=0$ iff $i=j$. Now, $\sbgrp{x_1,x_2,z}\simeq\Z^3$, and for the product to be the identity, we must always have $e=0$ in the exponent of $z$, meaning the from each $M^{h_i}$ only $S_i\times\{\uu_i\}$ is used. This implies that in the first coordinate $S_i$ are used in order, so $g\in M_1\cdots M_n$.
\end{proof}
A natural choice can be $u_k=v_k\simeq(N\cos(2\pi k/n),N\sin(2\pi k/n))$ for a large $N$ while making sure $\sum u_k=(0,0)$. We may also relax the requirement $\sum u_k=(0,0)$ and only require that $(0,0)$ is in the interior of the polygon, which implies $\sum\alpha_iu_i=0$ for some $\alpha_i\in\N_+$, and then the proof works by padding $g_{i,j}$ so that $l_i=c\alpha_i$ for some $c$. We can visualise $\sbgrp{x_1,x_2,z}=\{\xx^uz^e\mid u\in\Z^2,e\in\Z\}\simeq\Z^3$  as 3 dimensional space, where $e$ is the height, and $e=0$  is the floor. The vectors $\uu_i$ can then be seen as an inverted pyramid standing on its vertex (or as the lower part of a spinning dreidel), and each conjugation tilts the pyramid so that a single edge touches the ground (this "tilt" is linear but not orthogonal). For the sum of such conjugated vectors to be $(0,0)$ they must all touch the ground, so in each conjugate $M^{h_i}$ only one edge is active, allowing the use of $S_i$. We now easily get:
\begin{Theorem}
    Any product of f.g.\ submonoids of $G$ is a section of a submonoid of $G{\times}\heisen{5}{\times}(\heisen{3}/\sbgrp{z^e})$
\end{Theorem}
\begin{proof}
    Given such product $M_1\cdots M_n$ in $G$, by \Cref{thm:arbitraryprod2conjprod} there is a f.g.\ submonoid $M\le G{\times}\heisen{5}$ and $h_1,...,h_n\in\heisen{5}$ such that $M^{h_1}\cdots M^{h_n}\cap G=M$. We get that for every $g\in G$, $g\in M_1\cdots M_n$ iff $(g,h_1h_n^{-1})\in A\doteqdot Mh_1h_2^{-1}M\cdots h_{n-1}h_n^{-1}M$, and therefore $M_1\cdots M_n=\sect{A}{h_1h_n^{-1}}$. By \Cref{prop:MgMgM}, there is a f.g.\ submonoid $N\le G{\times}\heisen{5}{\times}(\heisen{3}/\sbgrp{z^e})$ and $v\in \heisen{3}/\sbgrp{z^e}$ such that $\sect{N}{v}=A$, and together we get $\sect{N}{(h_1h_n^{-1},v)}=M_1\cdots M_n$ as needed.
\end{proof}

\subsection{The second proof}
In this proof, instead of going first to a product of conjugate subsubmonoids and then to a submonoid, we unite both steps, giving a more efficient result. First we need a geometric construction.
\begin{Proposition}\label{vectors4H5}
For any $n$, there are vectors $u_1,...u_n,v_1,...,v_{n-1}\in\Z^2$ with the following properties:
\begin{enumerate}
\item\label{item:order} $\langle u_i,v_j\rangle <0$ when $i\le j$, and $\langle u_i,v_j\rangle >0$ when $i>j$
\item\label{item:uniquesum} If $\alpha_1,...,\alpha_{n-1}\in\N_0$ satisfy $\sum_{j=1}^{n-1}\alpha_jv_j=\sum_{j=1}^{n-1} v_j$, then $\alpha_1=\alpha_2=\ldots=\alpha_{n-1}=1$.
\item\label{item:zerosum}$\sum_{i=1}^nu_i=(0,0)$
\end{enumerate}
\end{Proposition}
A useful reformulation of \cref{item:order} is: If $u_i$ appears before $v_j$ in the word $u_1v_1u_2v_2...v_{n-1}u_n$, then $\langle u_i,v_j\rangle <0$, otherwise $\langle u_i,v_j\rangle >0$. 
\begin{proof}
Choose $0<a_1<b_1<a_2<b_2<\ldots<b_{n-1}<a_n$ such that $b_i>ib_{i-1}$ for every $i$, and define $v_j=(1,b_j)$ for $1\le j\le n-1$, $u_i=(a_i,-1)$ for $1<i<n$,  and finally $u_1=(-\sum a_i,0)$, and $u_n=(0,n-2)$.  \Cref{item:uniquesum} is satisfied by \Cref{prop:uniquesum}, \cref{item:zerosum} is satisfied by the choice of $u_1,u_n$, and \cref{item:order} can be checked directly.
\end{proof}
It may be easier to understand the conditions of \Cref{vectors4H5} by rotating $u_i$ by $\pi/2$ (so $(a_i,-1)\mapsto(1,a_i)$ in the specific vectors chosen in the proof). Then, instead of the inner product $\inner{u_i}{v_j}$ we would get the orientation of $u_i,v_j$, i.e.\ the sign of the determinant of $(u_i;v_j)$. This implies that scanning the vectors counterclockwise from $u_1$ we would see them in order $u_1v_1u_2v_2...v_{n-1}u_n$. In addition, \cref{item:uniquesum} can be shown to imply that $v_j$ are all in the same half-plane, while \cref{item:zerosum} implies that $u_i$ are not all on the same half-plane. This implies that the orientation of each 2 vectors of $u_1v_1u_2v_2...v_{n-1}u_n$ conforms with their order in the word, except $u_1u_n$ which have the opposite orientation, so the angle between them in the scan must exceed $\pi$ (it is $3\pi/2$ for the chosen vectors).

\begin{Proposition}
\label{prop:main_l1l1l1}
        For any $n$ there are elements $\uu_1,..,\uu_n,\vv_1,..,\vv_{n-1}\in \heisen{5}$ with the following properties:
        \begin{enumerate}
                \item $\uu_1^l\vv_1\uu_2^l\vv_2\cdots \vv_{n-1}\uu_n^l=\vv_1\cdots\vv_{n-1}$ for any $l\in\N$
                \item In any word $w$ without inverses in $\uu_1,..,\uu_n,\vv_1,..,\vv_{n-1}$ with value $\vv_1\cdots\vv_{n-1}$, the $\uu_i$'s appear in order: the last $\uu_i$ is located before the first $\uu_{i+1}$ for every $i$.
        \end{enumerate}
\end{Proposition}

\begin{proof}
        Let $u_1,u_2,...u_n,v_1,..,v_{n-1}$ be as in \cref{vectors4H5}, and define $e_{ij}=\inner{u_i}{v_j}$.
 Define $\uu_i=\xx^{u_i}z^{\sum_{i>j}e_{ij}}$ and  $\vv_i=\yy^{v_i}$. We first show:
\begin{equation}
\label{eq:l1l1l1}
\uu^l_1\vv_1\uu^l_2\vv_2...\vv_{n-1}\uu^l_n=\vv_1\vv_2...\vv_{n-1}    
\end{equation}

To see why \Cref{eq:l1l1l1} holds, we swap each $\uu_i$ left to the beginning of the word. We have $\vv_j\uu_i=\uu_i\vv_jz^{-e_{ij}}$ and $[\uu_i,\uu_j]=1$, so during the entire journey to the beginning of the word, a factor of $z^{-\sum_{j<i}e_{ij}}$ is added for each $\uu_i$. This cancels out the $z$ power in $\uu_i=\xx^{u_i}z^{\sum_{j<i}e_{ij}}$, and we are left with $\xx^{lu_1}\cdots \xx^{lu_n}\vv_1\cdots\vv_{n-1}$. But we have $\xx^{lu_1}\cdots \xx^{lu_n}=\xx^{l\sum_iu_i}=\xx^{(0,0)}=1$ by \cref{item:zerosum}.
Conversely, assume $w$ is a word in $\uu_1,..,\uu_n,\vv_1,..,\vv_{n-1}$ with value $\vv_1\cdots\vv_{n-1}$.
By applying \Cref{prop:uniquesum} to the projection of $w$ to $\heisen{5}/\sbgrp{x_1,x_2,z}\simeq\Z^2$, we get that each $\vv_j$ appears exactly once, not necessarily in order.  We now swap each $\uu_i$ to the beginning of $w$, just like before. Since each $\vv_j$ appears once in $w$, this contributes $z^{-\sum_{j\in J}e_{ij}}$  where $J$ is the set of indices of $\vv_j$ that appear before the $\uu_i$ we are swapping. Now, by \cref{item:order} we have $\sum_{j\in J}e_{ij}\le \sum_{j<i}e_{ij}$, and equality holds iff $J=\{1,...,i-1\}$. Therefore after moving $\uu_i=\xx^{u_i}z^{\sum_{i>j}e_{ij}}$ to the beginning of the word, we are left with $\xx^{u_i}z^e$ where $e=\sum_{j<i}e_{ij}-\sum_{j\in J}e_{ij}\ge 0$,  and $e=0$ iff $\uu_i$ appears after $\vv_1,...,\vv_{i-1}$ but before $\vv_i,...,\vv_{n-1}$. After moving all $\uu_i$ for every $i$ to the beginning, we get a word of the form $\xx^u\yy^vz^e$ for some $u,v\in\Z^2,e\ge 0$. Since elements of $\heisen{5}$ have a unique such form, and since the value of $w$ is assumed to be $\yy^{\sum_jv_j}$, we get that $e=0$, and therefore each $\uu_i$ appears after $\vv_1,...,\vv_{i-1}$ but before $\vv_i,...,\vv_{n-1}$. In particular, each $\uu_i$ appears before the single occurrence of $\vv_i$, and each $\uu_{i+1}$ appears after it, so the $\uu_i$ appear in order, as required.
\end{proof}

Now we are ready for the main theorem.
\begin{Theorem}
\label{thm:main_section}
    Let $M_1,...,M_n$ be f.g.\ submonoids of any group $G$. Then, the product $M_1\cdots M_n$ is a section of some  f.g.\ submonoid of $G{\times}\heisen{5}$.
\end{Theorem}
\begin{proof}
    Let $S_i\subseteq G$ be finite sets generating $M_i$, WLOG $1_G\in S_i$ for every $i$. Let $\uu_i,\vv_j$ be as in \Cref{prop:main_l1l1l1}, and define:
    \[T=\bigcup_{i\le n}(S_i\times\{\uu_i\})\cup\{(1_G,\vv_j)\mid j<n\}\]
and set $h=\vv_1\cdots\vv_{n-1}$. We claim that $\sect{\sbmon{T}}{h}=M_1\cdots M_n$.
For one side, let $g_i\in M_i$, we need to show $(g_1\cdots g_n,h)\in\sbmon{T}$. We can express $g_i=s_{i,1}s_{i,2}\cdots s_{i,l_i}$ where $l_i\in\N,s_{i,j}\in S_i$. By padding with $1_G\in S_i$ we may assume $l=l_1\cdots =l_n$. Since $(s_{i,j},\uu_i)\in T$ we get $(g_i,\uu_i^l)\in\sbmon{T}$. Therefore:
\[(g_1,\uu_1^l)(1_G,\vv_1)(g_2,\uu_2^l)(1_G,\vv_2)\cdots(g_n,\uu_n^l)=
(g_1\cdots g_n,\uu_1^l\vv_1\cdots\vv_{n-1}\uu_n^l)=(g_1\cdots g_n,h)\in\sbmon{T}\]
Where we used \Cref{prop:main_l1l1l1}. Conversely, assume $(g,h)=t_1\cdots t_l\in\sbmon{T}$. Looking at the second coordinate of $t_i$, we get by \Cref{prop:main_l1l1l1} that the $\uu_i$ appear in order, which implies that the in the first coordinate the elements of $S_i$ appear in order, so $g\in\sbmon{S_1}\cdots\sbmon{S_n}=M_1\cdots M_n$ as needed.
\end{proof}
\section{Decidability}
\begin{Theorem}
   \label{thm:reductionsH5}
   Let $G$ be any f.g.\ group.
    \begin{enumerate}
        \item Membership in a fixed product $M_1\cdots M_n$ of f.g.\ submonoids of $G$ can be reduced to membership in a fixed submonoid of $G{\times}\heisen{5}$
        \item Membership in all finite products of f.g.\ submonoids of $G$ can be reduced to the uniform submonoid membership problem in $G{\times}\heisen{5}$
    \end{enumerate}
\end{Theorem}
The proof is immediate from \Cref{thm:main_section}. We get a new, simple proof of a theorem of Romankov \cite{roman2023undecidability}:
\begin{Theorem}
    There is a group $G\in\nilptwo$ and a fixed f.g.\ submonoid $M\le G$, such that membership in $M$ is undecidable.
\end{Theorem}
\begin{proof}
    From \cite{konig2016knapsack}, there is a group $G\in\nilptwo$ with 4 subgroups $A_1,...,A_4$ such that membership in $A_1\cdots A_4$ is undecidable. By \Cref{thm:reductionsH5} we are done.
\end{proof}
We also get the following interesting converse to a reduction of Bodart \cite{bodart2024membership}:
\begin{Theorem}
    Let $k\ge 1$. The following problems are Turing equivalent:
    \begin{enumerate}
        \item The submonoid membership problem, uniformly for all  $G\in\nilptwo$ with $h([G,G])=k$.
        \item Deciding membership in products of submonoids, uniformly for all  $G\in\nilptwo$ with $h([G,G])=k-1$.
    \end{enumerate}
\end{Theorem}
\begin{proof}
One side is given by \cite{bodart2024membership}, and for the other side, we note that if $h([G,G])=k-1$ and $K=G{\times}\heisen{5}$ then $h([K,K])=h([G,G])+1=k$.
\end{proof}
\printbibliography

@misc{shafrir2024boundedgen,
      title={Bounded Generation of Submonoids of Heisenberg Groups}, 
      author={Doron Shafrir},
      year={2024},
      eprint={2405.05939},
      archivePrefix={arXiv},
      primaryClass={math.GR},
      url={https://arxiv.org/abs/2405.05939}
}

@misc{shafrir2024virtuallyab,
      title={Is decidability of the Submonoid Membership Problem closed under finite extensions?}, 
      author={Doron Shafrir},
      year={2024},
      eprint={2405.12921},
      archivePrefix={arXiv},
      primaryClass={math.GR},
      url={https://arxiv.org/abs/2405.12921}
}

@article{lohrey2015unitri,
  title={Rational subsets of unitriangular groups},
  author={Lohrey, Markus},
  journal={International Journal of Algebra and Computation},
  volume={25},
  number={01n02},
  pages={113--121},
  year={2015},
  publisher={World Scientific}
}

@article{konig2016knapsack,
  title={Knapsack and subset sum problems in nilpotent, polycyclic, and co-context-free groups},
  author={K{\"o}nig, Daniel and Lohrey, Markus and Zetzsche, Georg},
  journal={Algebra and Computer Science},
  volume={677},
  pages={138--153},
  year={2016},
  publisher={American Mathematical Society Providence, RI},
  url={https://arxiv.org/abs/1507.05145}
}

@article{roman2023undecidability,
  title={Undecidability of the submonoid membership problem for a sufficiently large finite direct power of the Heisenberg group},
  author={Roman'kov, Vitalii Anatol'evich},
  journal={Sibirskie Elektronnye Matematicheskie Izvestiya [Siberian Electronic Mathematical Reports]},
  volume={20},
  number={1},
  pages={293--305},
  year={2023},
  url={https://arxiv.org/abs/2209.14786}
}

@article{bodart2024membership,
  title={Membership problems in nilpotent groups},
  author={Bodart, Corentin},
  journal={arXiv preprint arXiv:2401.15504},
  year={2024},
  url={https://arxiv.org/abs/2401.15504}
}

@inproceedings{roman1999occurence,
  title={On the occurence problem for rational subsets of a group},
  author={Roman’kov, V},
  booktitle={International Conference on Combinatorial and Computational Methods in Mathematics},
  pages={76--81},
  year={1999}
}
\end{document}